\documentclass[12pt]{article}
\usepackage{xcolor}
\usepackage[colorlinks=true,urlcolor=blue,linkcolor=red,citecolor=magenta]{hyperref}
\usepackage{amsmath,amssymb,amsthm,amsfonts}
\usepackage{latexsym,graphicx,graphics,mathtools,paralist}
\usepackage{multicol}
\usepackage{color}
\usepackage{url}
\usepackage{euscript}
\usepackage{wrapfig}
\usepackage{caption}

\newtheorem{theorem}{Theorem}

\newtheorem{corollary}[theorem]{Corollary}

\newtheorem{definition}[theorem]{Definition}

\newtheorem{lemma}[theorem]{Lemma}

\newtheorem{proposition}[theorem]{Proposition}
\newtheorem{question}[theorem]{Question}
\theoremstyle{remark}
\newtheorem*{remark}{Remark}
\newtheorem*{remarks}{Remarks}
\newtheoremstyle{named}{}{}{\itshape}{}{\bfseries}{.}{.5em}{#1 \thmnote{#3}}
\theoremstyle{named}
\newtheorem*{namedques}{Question}

\newcommand{\R}{\mathbb R}

\newcommand{\Z}{\mathbb Z}
\newcommand{\one}{\mathbf 1}
\newcommand{\ii}{\mathbf i}

\newcommand{\trace}{\textnormal{trace}}
\newcommand{\id}{\textnormal{id}}

\begin{document}

\title{Tangent ray foliations and their associated outer billiards}
\author{Yamile Godoy, Michael Harrison and Marcos Salvai\thanks{This work was supported by Consejo Nacional de Investigaciones Cient\'ificas
y T\'ecnicas and Secretar\'{\i}a de Ciencia y T\'ecnica de la Universidad
Nacional de C\'ordoba.   Portions of this work were completed while the second author was in residence at the Institute for Advanced Study and supported by the National Science
Foundation under Grant No. DMS-1926686. 
}}
\maketitle

\begin{abstract}
Let $v$ be a unit vector field on a complete, umbilic (but not totally
geodesic) hypersurface $N$ in a space form; for example on the unit
sphere $S^{2k-1} \subset \mathbb{R}^{2k}$, or on a horosphere in hyperbolic
space. We give necessary and sufficient conditions on $v$ for the rays with
initial velocities $v$ (and $-v$) to foliate the exterior $U$ of $N$. We find and explore relationships among these
vector fields, geodesic vector fields, and contact structures on $N$.

When the rays corresponding to each of $\pm v$ foliate $U$, $v$ induces an
outer billiard map whose billiard table is $U$. We describe the unit vector
fields on $N$ whose associated outer billiard map is volume preserving.  Also
we study a particular example in detail, namely, when $N \simeq \mathbb{R}^3$
is a horosphere of the four-dimensional hyperbolic space and $v$ is the unit
vector field on $N$ obtained by normalizing the stereographic projection of
a Hopf vector field on $S^{3}$. In the corresponding outer billiard map we
find explicit periodic orbits, unbounded orbits, and bounded nonperiodic orbits. We
conclude with several questions regarding the topology and geometry of bifoliating vector fields and the dynamics of their associated outer billiards.
\end{abstract}


\medskip


\section{Motivation}
\label{sec:motivation}

We begin with a simple observation: given one of the two unit tangent vector
fields on the unit circle $S^1 \subset \mathbb{R}^2$, the corresponding
tangent rays foliate the exterior of $S^1 \subset \mathbb{R}^2$. This leads
to the following question.

\begin{question}
\label{ques:main} What conditions on a unit tangent vector field on $%
S^{2k-1}\subset \mathbb{R}^{2k}$ guarantee that the corresponding tangent
rays foliate the exterior of $S^{2k-1}$?
\end{question}

Prototypical examples arise from great circle fibrations of $S^{2k-1}$, for
example, the standard Hopf fibrations. Indeed, each great circle $C$ can be
written as the intersection of $S^{2k-1}$ with a $2$-plane $P_C$. Since no
two great circles intersect, the exterior of $S^{2k-1}$ can be foliated using the
corresponding collection of planes, and the exterior of $C$ in each such plane $P_C$ is foliated by
the rays tangent to $C$. In this way, Question \ref{ques:main} is related to
the study of geodesic fibrations of spheres.

Our interest in Question \ref{ques:main} also stems from the fact that \emph{%
skew} geodesic fibrations of Euclidean space (that is, fibrations of $%
\mathbb{R}^n$ by nonparallel straight lines) can only exist for odd $n$.  
Therefore the tangent ray foliations we study here may be viewed as an even-dimensional counterpart to the odd-dimensional phenomenon of skew line fibrations.

Additional motivation for Question \ref{ques:main} arises from the study of 
\emph{outer billiards}. As suggested by its name, outer billiards is played
outside a smooth closed strictly convex curve $\gamma \subset \mathbb{R}^2$,
and can be easily defined as follows: fix one of the two unit tangent vector
fields $v$ on $\gamma$, and observe that the corresponding tangent rays
foliate the exterior of $\gamma$. In particular, for each point $x$ outside
of $\gamma$, there exists exactly one tangent ray passing through $x$, and
the outer billiard map $B$ is defined by reflecting $x$ about the point of
tangency (see Figure \ref{fig:outerplane}). Originally popularized by Moser 
\cite{Moser1, Moser2}, who studied the outer billiard map as a crude model
for planetary motion, the outer billiard has since been studied in a number
of contexts; see \cite{DT, ST1,ST2,ST4} for surveys.

\begin{figure}[ht!]
\centerline{
\includegraphics[width=2.5in]{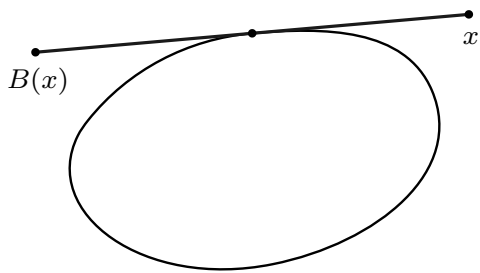}
}
\caption{The outer billiard map in the plane}
\label{fig:outerplane}
\end{figure}

When attempting to define outer billiards in higher dimensional Euclidean
space, one encounters the following issue: given a smooth closed strictly
convex hypersurface $N$, there are too many tangent lines passing through
each point $x$ outside of $N$, and so it is not obvious how to define outer
billiards with respect to $N$. In \cite{ST3}, Tabachnikov resolved this
issue in even-dimensional Euclidean space $\mathbb{R}^{2k}$, endowed with
the standard symplectic structure $\omega$, by appealing to the \emph{%
characteristic line bundle} $\xi \coloneqq \operatorname{ker}(\omega|_N)$ on $N$.
This line bundle $\xi$ has two unit sections $\pm v$, and each has the
property that the corresponding tangent rays foliate the exterior of $N$.
Thus choosing either $v$ or $-v$ yields a well-defined, invertible outer
billiard map on the exterior of $N$; moreover, Tabachnikov proved that this
outer billiard map is a symplectomorphism with respect to $\omega$. As an
example, when $N = S^{2k-1}$, the characteristic lines are tangent to the
great circles of the Hopf fibration.


We emphasize that the essential ingredient in the outer billiard
construction is that the tangent rays corresponding to both $\pm v$ foliate
the exterior of $N$. Thus an additional motivation for Question \ref%
{ques:main} is the construction of outer billiard systems which may exhibit
interesting dynamical properties.

\section{Statement of results}
\label{sec:statement}
Our main result not only provides a complete answer to Question \ref{ques:main}, but it also applies to the more general situation in which the ambient Euclidean space is replaced by the $(n+1)$-dimensional space form $M_{\kappa}$ of constant sectional curvature $\kappa$, and the role of the sphere $S^n \subset \mathbb{R}^{n+1}$ is played by a complete umbilic hypersurface $N \subset M_{\kappa}$ which is not totally geodesic (recall that $N \subset M_{\kappa}$ is \textbf{umbilic} if all of its principal curvatures are equal).  In particular, the sectional curvature of $N$ is constant and strictly larger than $\kappa$ (as a corollary of the Gauss Theorem, see for instance Remark 2.6 in Chapter 6 of \cite{docarmo}).

Specifically, we will consider Question \ref{ques:main} in the following five settings:
\begin{compactenum}[(1)]
\item $\kappa = 0$: $M_0 = \R^{n+1}$, and $N$ is an $n$-dimensional round sphere,
\item $\kappa > 0$: $M_\kappa$ is a round sphere and $N$ is a geodesic sphere which is not a great sphere; that is, its curvature is larger than $\kappa$, 
\item $\kappa < 0$: $M_{\kappa}$ is the $(n+1)$-dimensional hyperbolic space $H^{n+1}$.  It is convenient to describe the umbilics in the upper-half space model $\left\{(x_0,\dots,x_n) \in \R^{n+1} \mid x_0 > 0\right\}$ with the metric $ds^2 = (dx_0^2 + \cdots + dx_n^2)/(-\kappa x_0^2)$.  The umbilic hypersurfaces are then the intersections with $H^{n+1}$ of hyperspheres and hyperplanes in $\R^{n+1}$, and there are three possibilities for $N$:
\begin{compactenum}[(3a)]
\item $N$ is a geodesic sphere,
\item $N$ is a horosphere, which is congruent by an isometry to $x_0 = 1$,
\item $N$ is congruent by an isometry to the intersection of $H^{n+1}$ with a hyperplane through the origin which is not orthogonal to the hyperplane $x_0 = 0$.
\end{compactenum}
\end{compactenum}

Observe that in each case, $N$ is diffeomorphic to a sphere or to Euclidean space.  



Next, we define the \textbf{exterior} $U$ of $N$ as the connected component of $M_{\kappa }-N$ into which the opposite of the mean curvature vector field of $N$ points, except that for $\kappa > 0$, we further restrict $U$ to be the zone between $N$ and its antipodal image.  The exterior $U$ in each of the five situations is depicted in Figure \ref{fig:u}.

\begin{figure}[ht!]
\centerline{
\includegraphics[width=6in]{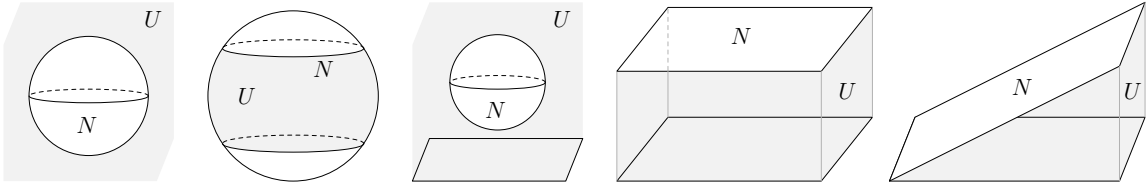}
}
\caption{The exterior $U$, from left to right, for $\kappa = 0$, $\kappa >0$, and $\kappa < 0$} 
\label{fig:u}
\end{figure}


Given $u\in TM_{\kappa}$, we denote by $\gamma _{u}$ the unique
geodesic in $M_{\kappa } $ with initial velocity $u$, and we write $%
T_{\kappa }=\infty $ for $\kappa \leq 0$ and $T_{\kappa}=\pi/\sqrt{\kappa}$ for $\kappa >0$. Note that a unit speed geodesic of the sphere tangent to $N$ at $p$ travels through $U$ and hits $-p \in -N$ at time $T_\kappa$.

\begin{definition}
Let $v$ be a smooth unit vector field on a complete umbilic hypersurface $N$
of $M_{\kappa }$ which is not totally geodesic, and let $U$ be the exterior
of $N$ as defined above. We say that $v$: 
\begin{compactitem}
\item \textbf{forward foliates $U$} if the geodesic rays $\gamma _{v\left(p\right)}\left(0,T_{\kappa }\right)$ are the leaves of a smooth foliation of $U$,
\item \textbf{backward foliates $U$} if the geodesic rays $\gamma _{-v\left(p\right)}\left(0,T_{\kappa }\right)$ are the leaves of a smooth foliation of $U$,
\item \textbf{bifoliates $U$} if $v$ both forward and backward foliates $U$.
\end{compactitem}
\end{definition}


While a forward or backward foliation induces a dynamical system which may not be time-reversible, a bifoliation induces a smooth, invertible outer billiard map on $U$ (see (\ref{eqn:outerbilliarddef}) for the definition), justifying our interest in this specific notion.

With all of the terminology introduced above, Question \ref{ques:main} admits the following general formulation:

\begin{namedques}[$\mathbf{1^\prime}$]
Let $U$ be the exterior of a complete umbilic not totally geodesic
hypersurface $N$ in $M_{\kappa }$. What conditions on a smooth unit vector
field $v$ on $N$ guarantee that $v$ bifoliates $U$?
\end{namedques}


We denote by $\nabla$ the Levi-Civita connection on $N$. Observe that since $%
v$ has unit length, the image of $\left( \nabla v \right)_p$ is orthogonal to $%
v(p)$ for every $p \in N$. In particular, $\left( \nabla v\right) _{p}$ is
singular and preserves the subspace $v(p)^\perp \subset T_pN$.

We now provide a complete classification of bifoliating vector fields.

\begin{theorem}
\label{bifoliate} Let $U$ be the exterior of a complete umbilic not totally
geodesic hypersurface $N$ in $M_{\kappa }$ and let $v$ be a smooth unit
vector field on $N$. Then the assertions below are equivalent: 
\begin{compactenum}[(a)]
\item The vector field $v$ bifoliates $U$.
\item For each $p\in N$, any \emph{real} eigenvalue $\lambda $ of the
restriction of the operator $\left( \nabla v\right) _{p}$ to $v(p) ^{\perp }\subset T_{p}N$ satisfies 
$\lambda ^{2} + \kappa \leq 0$.

\item For each $p\in N$ and any \emph{real} eigenvalue $\lambda $ of $%
\left( \nabla v\right)_{p}$, the following condition holds:
\begin{compactenum}[(i)]

\item for $\kappa = 0$\emph{:} $\lambda =0$;

\item for $\kappa > 0$\emph{:} $\lambda =0$ with algebraic multiplicity one;

\item for $\kappa < 0$\emph{:} $\lambda ^{2}\leq -\kappa$.
\end{compactenum}
\end{compactenum}
\end{theorem}

\begin{remark}
Although Theorem \ref{bifoliate} is written without explicit topological
restrictions on $N$, the existence of $v$ implies that $N$ is not an even-dimensional sphere.
\end{remark}

\begin{remark} We emphasize a surprising feature of Theorem \ref{bifoliate}: the global condition that $v$ bifoliates $U$ is characterized by an infinitesimal condition on $v$, which one might expect to only guarantee a smooth foliation locally.
\end{remark}

\begin{remark} Looking at the proof of Theorem \ref{bifoliate} one can easily deduce conditions for a unit vector field on $N$ to forward or backward foliate the exterior.
\end{remark}

At the beginning of Section \ref{sec:motivation} we observed that great
circle fibrations of odd-dimensional spheres induce bifoliating
vector fields. This statement persists in the general setting. A \textbf{%
geodesic vector field} on $N$ is a unit vector field on $N$ whose integral
curves are geodesics. Using the explicit criterion of Theorem \ref{bifoliate}%
, we show that geodesic vector fields are bifoliating, but not all
bifoliating vector fields are geodesic.

\begin{theorem}
\label{thm:geodesic} Let $U$ be the exterior of a complete umbilic not
totally geodesic hypersurface $N$ in $M_{\kappa }$. 
\begin{compactenum}[(a)]
\item If a smooth unit vector field $v$ on $N$ is geodesic, then $v$ is bifoliating.
\item If $n > 1$, there exists a smooth bifoliating vector field on $N$ which is not geodesic.
\end{compactenum}
\end{theorem}


The latter statement is credible for the following simple reason: for $\kappa \geq 0$, the vector
field $v$ tangent to the Hopf fibration on $N = S^{2k+1}$ is both bifoliating and geodesic,
but an arbitrarily small perturbation can ruin the symmetry of the integral
curves while maintaining the open condition that the restriction of the
operator $\left( \nabla v\right) _{p}$ to $v(p)^{\perp }\subset T_{p}N$ has
no real eigenvalues. This idea motivates the explicit examples which we
provide in the proof of Theorem \ref{thm:geodesic}.  

A geodesic vector field on a Riemannian manifold $M$ determines an oriented 
\emph{geodesic foliation} of $M$. In this language, Theorem \ref%
{thm:geodesic} says that a geodesic vector field on $N$ determines \emph{both%
} a geodesic foliation of $N$ itself and a bifoliation of the exterior $U$,
but that some bifoliations of $U$ arise from vector fields on $N$ which do
not determine geodesic foliations.

Geodesic foliations of the space forms $M_\kappa$ are of interest in their
own right and have been studied extensively. Great circle fibrations of $S^3$
are characterized in \cite{GW} and studied in higher dimensions in \cite%
{GWY,McKay}.  Geodesic foliations of $\mathbb{R}^{3}$ and hyperbolic space $%
H^{3}$ have been characterized in \cite{GSMZ,HarrisonAGT,HarrisonMZ,OvsienkoTabachnikov, SalvaiBL}. In \cite%
{Gluck}, Gluck proved that the plane field orthogonal to a great circle
fibration of $S^3$ is a tight contact structure. The relationship between
line fibrations of $\mathbb{R}^3$ and (tight) contact structures was studied
in \cite{BeckerGeiges,HarrisonBLMS,HarrisonAGT}. We show that a similar
relationship exists for bifoliating vector fields.

\begin{theorem}
\label{thm:contact} Let $v$ be a smooth bifoliating unit vector field on a sphere $N \subset S^{4}$. Then the $1$-form dual to $v$ is a contact
form.
\end{theorem}

\begin{remarks}
We make several remarks to contextualize Theorem \ref{thm:contact}. 
\begin{compactenum}[1)]
\item By Theorem \ref{bifoliate}, every bifoliating vector field $v$ on $N \subset S^4$ satisfies a certain condition, namely, that $\left( \nabla v\right) _{p}$ has rank $2$ for all $p\in S^{3}$.  Bifoliating vector fields on $S^3 \subset \R^4$ do not necessarily satisfy this nondegeneracy condition; see Proposition \ref{prop:bizero} for an example.  However, Theorem \ref{thm:contact} and its proof do hold for smooth bifoliating vector fields on $S^3 \subset \R^4$ if the condition is added as a hypothesis.

\item Theorem \ref{thm:contact} fails in higher dimensions for $\kappa \geq 0$. In \cite{GluckYang}, Gluck and Yang construct examples of unit vector fields on $S^{n}$, for odd $n\geq 5$, which determine great circle fibrations (and thus are bifoliating by Theorem \ref{thm:geodesic}) for which the dual form is not contact.

\item Theorem \ref{thm:contact} fails for the horosphere in hyperbolic space, since a constant vector field is bifoliating, but the dual $1$-form is not contact.

\item We do not know if the contact structures in Theorem \ref{thm:contact} are tight; see Question \ref{ques:defret} for discussion.
\end{compactenum}
\end{remarks}

We now turn our attention to the dynamical properties exhibited by
bifoliating vector fields. A bifoliating vector field $v$ on $N$
induces an \textbf{outer billiard map} $B \colon U\rightarrow U$
defined by 
\begin{equation}
\label{eqn:outerbilliarddef}
B(\gamma _{v(p)}(-t))=\gamma _{v(p)}(t), \hspace{.25in} (p,t) \in N \times
(0,T_{\kappa}). 
\end{equation}
See Figure \ref{fig:hyper} for a depiction in hyperbolic space.

\begin{figure}[ht!]
\centerline{
\includegraphics[width=3.75in]{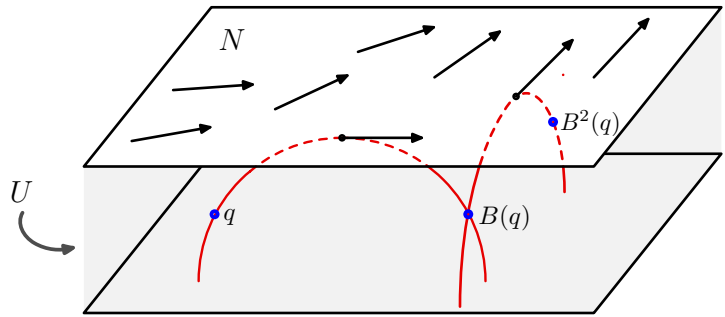}
}
\caption{Two iterations of the outer billiard map induced by a bifoliating vector field on a horosphere $N$,  seen in the upper-half space model of hyperbolic space}
\label{fig:hyper}
\end{figure}

Typically, the dynamics of billiard systems can be
studied via their symplectic properties. However, the outer billiard systems
induced by bifoliating vector fields are, in general, not symplectic. It
seems plausible that different techniques are necessary for a careful study of their
dynamics.

On the other hand, we study a particular example in detail, for which $%
N \simeq \mathbb{R}^3$ is a horosphere in $H^4$ and $v$ is the unit vector
field obtained by normalizing the stereographic projection of the Hopf
vector field on $S^3$. For the associated outer billiard map we find explicit periodic orbits, unbounded orbits, and bounded nonperiodic orbits; see Proposition \ref{prop:hopfr3}.

We conclude by studying the relationship between volume-preservation of $B$
and the characteristic polynomials $Q_p(s) \coloneqq \det \left( s~\text{id}%
_{T_{p}N}-\left( \nabla v\right) _{p}\right) $ of $\left( \nabla v\right)
_{p}$ for $p \in N$.

\begin{theorem}
\label{volume} Let $U$ be the exterior of a complete umbilic not totally
geodesic $n$-dimensional hypersurface $N$ in $M_{\kappa }$ and let $v$ be a
smooth bifoliating unit vector field on $N$.  Then the associated outer billiard map $B \colon U \to U$ be preserves volume if and
only if, for all $p \in N$ and $s \in \mathbb{R}$, 
\begin{equation*}
Q_{p}(-s) =\left( -1 \right)^{n}Q_{p}(s); 
\end{equation*}
that is, for all $p$, the parity of $Q_{p}$ coincides with the parity of $n$.
\end{theorem}


\begin{corollary}
\label{cor:volume} Let $U$ be the exterior of a complete umbilic not totally
geodesic $n$-dimensional hypersurface $N$ in $M_{\kappa }$,  and let $B \colon U \to U$ be the outer billiard map associated to a bifoliating unit vector field $v$ on $N$.

\begin{compactenum}[(a)]
\item If the map $B$ preserves volume, then $\operatorname{div}v$ vanishes identically, that is, the flow of $v$ preserves the volume of $N$.  If additionally 
$n$ is even, then the restriction of $\left( \nabla v\right) _{p}$ to $v(p)^{\perp }$ is singular for all $p\in N$.
\item If $n=2$ (which implies that $\kappa < 0$, $N$ is diffeomorphic to $\R^2$, and its intrinsic metric has constant Gaussian curvature $k$ with $\kappa < k \leq 0$), then the following are equivalent:
\begin{compactenum}[(i)]
\item the map $B$ preserves volume,
\item $v$ is orthogonal to a geodesic foliation of $N$,
\item $\operatorname{div} v$ vanishes identically.
\end{compactenum}
If $N \simeq \R^2 \subset H^3$ is a horosphere, the above conditions are equivalent to $v$ being constant. 

\item If $n=3$, then the map $B$ preserves volume if and only if $\operatorname{div}v$ vanishes identically.  If additionally $N = S^3$ and $v$ is geodesic (that is, $v$ determines a great circle fibration), then $B$ preserves volume if and only if $v$ is Hopf.
\end{compactenum}
\end{corollary}

We comment that the normalization of a nowhere vanishing Killing field on $N \subset H^3$ provides an example of a vector field satisfying the conditions in (b).

\section{Preliminaries on Jacobi fields}

Here we provide a brief review of Jacobi fields, which arise naturally when
studying variations of geodesics, and thus play a central role in the proofs
of the main theorems. A more thorough treatment can be
found in any standard Riemannian geometry text, for example \cite{docarmo}.

Let $M$ be a complete Riemannian manifold and let $\gamma $ be a complete
unit speed geodesic of $M$. A \textbf{Jacobi field} $J$ along $\gamma $ is
by definition a vector field along $\gamma $ arising via a variation of
geodesics as follows: Let $\delta >0$ and $\phi \colon \mathbb{R}\times (
-\delta ,\delta) \rightarrow M$ be a smooth map such that $s\mapsto
\phi(s,t)$ is a geodesic for each $t\in ( -\delta ,\delta) $ and $\phi ( s,0) =\gamma ( s) $ for all $s$.
Then
\[
J( s) =\left. \frac{d}{dt}\right\vert _{0}\phi ( s,t).
\]

When $M = M_\kappa$, it is well known that
Jacobi fields along a unit speed geodesic $\gamma $ and orthogonal to $%
\gamma ^{\prime }$ are exactly those vector fields $J$ along $\gamma $ with $%
\left\langle J,\gamma ^{\prime }\right\rangle =0$ satisfying the equation 
\begin{align}
\frac{D^{2}J}{ds^{2}}+\kappa J=0.  \label{equationJNormal}
\end{align}

If the initial conditions of $J$ are $J( 0) =a\gamma ^{\prime
}( 0) +u$ and $\frac{DJ}{dt}( 0) =w$, 
with $a\in \mathbb{R}$ and $u,w$
orthogonal to $\gamma ^{\prime }( 0) $, then%
\begin{equation}
J( t) =a\gamma ^{\prime }( t) +c_{\kappa}( t)
U( t) +s_{\kappa}( t) W( t),
\label{Jacok}
\end{equation}%
where $U,W$ are parallel vector fields along $\gamma $ with $U\left(
0\right) =u$ and $W\left( 0\right) =w$, 
\begin{equation}
s_{\kappa}\left( t\right) =\left\{ 
\begin{array}{ll}
\sin \left( \sqrt{\kappa}t\right) /\sqrt{\kappa} & \text{if }\kappa>0, \\ 
t & \text{if }\kappa=0,  \\ 
\sinh \left( \sqrt{-\kappa}t\right) /\sqrt{-\kappa} \ \ & \text{if }\kappa<0,
\end{array}%
\right. \text{\ \ \ \ \ \ \ \ \ and }\hspace{0.25cm}c_{k}(t)=s_{k}^{\prime
}(t).  \label{s_k}
\end{equation}%
Also, we call $\cot _{\kappa}=c_{\kappa}/s_{\kappa}$. Expression (\ref{Jacok}) will
allow us to perform most computations without having to resort to
coordinates of $M_{\kappa }$ or a particular model of it. 

\begin{lemma}
\label{t_o} Let $\lambda \in \mathbb{R}$. If $\lambda ^{2}+\kappa>0$, then the
equation $c_{\kappa}( t) -\left\vert \lambda \right\vert s_{\kappa}(
t) =0$ has a solution in the interval $\left( 0,T_{\kappa}\right) $.
\end{lemma}

\begin{proof}
The assertion follows from the fact that for $\kappa=0$, $\kappa>0$ and $\kappa<0$, the
equation reads $\left\vert \lambda \right\vert t=1$,\ $\cot \left( \sqrt{\kappa}%
t\right) =\left\vert \lambda \right\vert /\sqrt{\kappa}$\ and $\coth \left( \sqrt{%
-\kappa}t\right) =\left\vert \lambda \right\vert /\sqrt{-\kappa}$, respectively.
\end{proof}

\begin{lemma}
\label{Jfol}Let $M$ be a Riemannian manifold and let $\phi :\left(
-\varepsilon ,\varepsilon \right) \times \left( 0,T\right) \rightarrow M$ be
a geodesic variation, that is, $\gamma _{s}:\left( 0,T\right) \rightarrow M$
is a unit speed geodesic in $M$ for all $s\in \left( -\varepsilon
,\varepsilon \right) $, where $\gamma _{s}\left( t\right) =\phi \left(
s,t\right) $. Suppose that $V$ is a unit vector field on $M$ such that $%
V\left( \gamma _{s}\left( t\right) \right) =\gamma _{s}^{\prime }\left(
t\right) $ for all $s,t$. Then the Jacobi field along $\gamma _{0}$
associated with $\phi $ vanishes at some $t_{o}\in \left( 0,T\right) $ only
if it is identically zero.
\end{lemma}

\begin{proof}
We call $\beta \left( s\right) =\phi \left( s,t_{o}\right) .$ We have that $%
0=J\left( t_{o}\right) =\beta ^{\prime }\left( 0\right) $. We compute 
\begin{eqnarray*}
\frac{DJ}{dt}(t_{o}) &=&\left. \frac{D}{\partial t}\right\vert
_{t_{o}}\left. \frac{\partial }{\partial s}\right\vert _{0}\phi \left(
s,t\right) =\left. \frac{D}{\partial s}\right\vert _{0}\left. \frac{\partial 
}{\partial t}\right\vert _{t_{o}}\gamma _{s}\left( t\right) \\
&=&\left. \frac{D}{\partial s}\right\vert _{0}\gamma _{s}^{\prime }\left(
t_{o}\right) =\left. \frac{D}{\partial s}\right\vert _{0}V(\beta (s))=\nabla
_{\beta ^{\prime }(0)}V=0.
\end{eqnarray*}%
Since $J$ is the solution of a second order differential equation, $J\equiv
0 $, as desired.
\end{proof}

\section{Bifoliations and the proof of Theorem \protect\ref{bifoliate}}

We now return to the situation in which $N$ is a complete umbilic hypersurface of $M_{\kappa }$ which is not totally geodesic.   Recall that a complete list of such pairs $(M_\kappa,N)$ was given at the beginning of Section \ref{sec:statement}.   We denote by $
\overline{\nabla}$ and $\nabla $ the Levi-Civita connections of $M_{\kappa }$ and 
$N$, respectively. The Gauss formula in this case is given by 
\begin{equation}
\overline{\nabla}_{X}Y=\nabla _{X}Y+\langle X,Y\rangle _{\kappa }H, \hspace{.25in} X,Y\in 
\mathfrak{X}(N),
\label{eqn:gauss}
\end{equation}%
where $H$ is the mean curvature vector field on $N$.  Unless otherwise stated, geodesics are always in $M_{\kappa}$.   

Now let $v$ be a unit vector field on $N$.  For $t\in \left( 0,T_{\kappa }\right) $, we define
\[
f_t \colon N \to M_\kappa, \hspace{.5in} f_{t}(p) =\gamma _{v\left( p\right) }(t).
\]
We first concentrate on the image of $f_t$.

Let $G_{\kappa }=$ Iso$_{o}( M_{\kappa }) $ be the identity
component of the isometry group of $M_{\kappa }$ and let $L=\left\{ g\in
G_{\kappa }\mid g( N) =N\right\}$, which is isomorphic to Iso$_{o}(N)$.   When $N$ is a sphere, there are one or two trivial orbits of $L$ (depending on the ambient space); otherwise, the orbits of $L$ are the parallel hypersurfaces to $N$, which are well-known to be umbilic. 

\begin{lemma} For each $t \in (0,T_\kappa)$, the image of $f_t$ is contained in exactly one hypersurface parallel to $N$.
\end{lemma}




\begin{proof}
Let $p, q \in N$.  
Since $N$ is extrinsically two-point
homogeneous, there exists $h\in L$ such that $h(p)=q$ and $dh_{p}\,v(p)=v(q)$.  So,
\[
f_{t}( q) =\gamma _{v\left( q\right) }(t)=\gamma _{dh_{p}v(
p) }(t)=h(\gamma _{v\left( p\right) }(t))=h( f_{t}( p)
),
\]
as desired.
\end{proof}

Now for $t \in (0,T_\kappa)$, we define $N_t$ as the hypersurface parallel to $N$ which contains the image of $f_t$.  Since $N_t$ is embedded,  the map $f_t \colon N \to N_t$ is smooth.  We next compute its differential.

\begin{lemma}
\label{Jf}Let $v$ be a unit vector field on $N$. Let $p\in N$ and $%
x=av\left( p\right) +u\in T_{p}N$, with $a\in \mathbb{R}$ and $u \perp v\left(
p\right) $. Then%
\begin{equation}
\left( df_{t}\right) _{p}\left( x\right) =\tau _{0}^{t}\left( av(p)+\left(
c_{\kappa }(t)u+s_{\kappa }(t)\nabla _{x}v\right) +s_{\kappa }(t)aH\left(
p\right) \right),   \label{dft}
\end{equation}%
where $\tau _{0}^{t}$ denotes the parallel transport along $\gamma _{v(p)}$
between $0$ and $t$.
\end{lemma}

\begin{proof}
Let $\alpha $ be a smooth curve in $N$ with $\alpha ^{\prime }(0)=x$. Let $J$
be the Jacobi field along $\gamma _{v(p)}$ associated with the geodesic
variation $\left( s,t\right) \mapsto \gamma _{v(\alpha (s))}(t)$. We have $J(0)=x$ and, with an argument similar to that in the proof of Lemma \ref{Jfol}, 
\begin{equation*}
\frac{DJ}{dt}(0)=\overline{\nabla}_{x}v.
\end{equation*}%
So $J$ depends only on $x$ and we call it $J_{x}$. By (\ref{Jacok}) and (%
\ref{eqn:gauss}),
\begin{equation}
J_{x}(t)=a\gamma _{v(p)}^{\prime }(t)+c_{\kappa }(t)\tau
_{0}^{t}(u)+s_{\kappa }(t)\tau _{0}^{t}(\nabla _{x}v+aH( p) ).  \label{Jx}
\end{equation}%
Then
\begin{equation*}
(df_{t})_{p}\left( x\right) =\left. \frac{d}{ds}\right\vert _{0}f_{t}\left(
\alpha \left( s\right) \right) =J_{x}(t)
\end{equation*}
and the lemma follows.
\end{proof}

\begin{proposition}
\label{f_tDifeo}If for each $p\in N$ any \emph{real} eigenvalue $\lambda $
of the restriction of the operator $\left( \nabla v\right) _{p}$ to $v\left(
p\right) ^{\perp }\subset T_{p}N$ satisfies $\lambda ^{2}+\kappa \leq 0$,
then $f_{t}:N\rightarrow N_{t}$ is a diffeomorphism.
\end{proposition}

\begin{proof}
We prove first that $f_{t}$ is a local diffeomorphism.  To argue the contrapositive,  suppose that $(df_{t})_{p}( x) =0$ for some $x \neq 0$.
Since the three terms on the right hand side of (\ref{dft}) are
pairwise orthogonal, we conclude that $a=0$ (in particular, $x=u$) and $c_{\kappa
}(t)x+s_{\kappa }(t)\nabla _{x}v=0$.   Hence,%
\begin{equation*}
\nabla _{x}v=-\cot _{\kappa }( t) x\text{,}
\end{equation*}%
that is, $x\in v(p)^{\perp }$ is an eigenvector of $(\nabla v)_{p}$ with
eigenvalue $-\cot _{\kappa }t$.   Together with the observation that $\cot
_{\kappa }^{2}(t)+\kappa >0$ holds for all $\kappa$, this completes the argument.

Next we prove that $f_{t}:N\rightarrow N_{t}$ is a diffeomorphism.   For $%
N=S^{1}$, the assertion is clear.   If $N$ is
a sphere different from the circle, then, by compactness, $f_{t}$ is a
covering map, which must be a bijection since $N$ is simply connected. 

If $N$ is not a sphere, then $\kappa < 0$ and both $N$ and $N_t$ are diffeomorphic to $\R^n$.  To see that $f_t$ is a diffeomorphism, we apply a ``global inverse'' result of Hadamard (see \cite[Theorem 6.2.8]{KrantzParksBook} or \cite{Gordon}), which asserts that a smooth proper local diffeomorphism from $\R^n$ to $\R^n$ is a diffeomorphism; here \emph{proper} means that the preimage of every compact set is compact. 
Properness of $f_t$ follows from the facts that $f_t$ displaces each point by distance $t$ (in $M_\kappa$) and that $N$ is properly embedded in $M_\kappa$.  Hence $f_t$ is a diffeomorphism.
\end{proof}


We are now prepared to prove Theorem \ref{bifoliate}.

\begin{proof}[Proof of Theorem \protect\ref{bifoliate}]
``(a) $\Rightarrow $ (b)'' 
Let $\lambda \in \mathbb{R}$ and suppose there exists a nonzero tangent
vector $u\in v\left( p\right) ^{\perp }\subset T_{p}N$ such that $\nabla
_{u}v=\lambda u$.   We consider the Jacobi vector field 
\begin{equation*}
J(t)=\left. \frac{d}{ds}\right\vert _{0}\gamma _{ v(\alpha
(s))}( t) \text{,}
\end{equation*}%
where $\alpha $ is a smooth curve in $N$ with $\alpha ^{\prime }\left(
0\right) =u$. By (\ref{eqn:gauss}) we have that $\nabla _{u}v=\overline{\nabla}%
_{u}v$ and so, $J(0)=u$ and $\frac{D J}{dt}(0) = \lambda u$. Using (\ref{Jacok}), we obtain 
\begin{equation*}
J(t)=(c_{\kappa }(t)+ \lambda s_{\kappa }(t))U(t)%
\text{,}
\end{equation*}%
where $U$ is the parallel vector field along $\gamma_{v(p)} $ such that $U(0)=u$.

Now $J$ is associated with the variation given by the geodesic rays of the forward
foliation of $U$, and Lemma \ref{Jfol} implies that $J$ does not vanish for any $t \in (0, T_\kappa)$.  Now if $\lambda \leq 0$, Lemma \ref{t_o} applies and yields $\lambda^2 + \kappa \leq 0$, as desired.  In case $\lambda > 0$, we instead consider the Jacobi field associated with the variation given by the geodesic rays $\gamma_{-v(\alpha(s))}(t)$ of the backward foliation of $U$ and proceed similarly.

\bigskip

``(b) $\Rightarrow $ (a)''  To verify that the vector field $v$ bifoliates $U$, we show that $F_+$ and $F_-$ are diffeomorphisms, where $F_+$ and $F_-$ are the restrictions to $N\times (0,T_{\kappa })$ and $N\times (-T_{\kappa },0)$, respectively, of the smooth function
\[ F \colon N\times (-T_{\kappa },T_{\kappa
})\rightarrow M_{\kappa }, \ \ \ \ 
F(p,s)=\gamma _{v\left( p\right) }(s).
\]
We deal only with $F_+$, since the case of $F_-$ is analogous.

The map $F_+$ is a bijection by Proposition \ref%
{f_tDifeo}, since $f_{t}(N)=N_{t}$ for all $t\in \left( 0,T_{\kappa }\right)$,
and the umbilic hypersurfaces $N_t$ foliate $U$.

Now we check that $(dF_+)_{(p,t)}$ is an isomorphism. By Proposition \ref%
{f_tDifeo}, it sends $T_{p}N\times \left\{ 0\right\} $ isomorphically to $%
T_{f_{t}\left( p\right) }N_{t}=\left( df_{t}\right) _{p}\left( T_{p}N\right) 
$. Hence, it suffices to show that%
\begin{equation*}
(dF_+)_{(p,t)}\left( 0,\left. \tfrac{\partial }{\partial s}\right\vert
_{t}\right) =\gamma _{v\left( p\right) }^{\prime }\left( t\right) \not\in
T_{f_{t}\left( p\right) }N_{t}\text{.}
\end{equation*}
Assume otherwise,  so that $\gamma _{v\left( p\right) }^{\prime }\left(
t\right) =\left( df_{t}\right) _{p}\left( x\right) $ for some $x\in T_{p}N$.
Applying $\left( \tau _{0}^{t}\right) ^{-1}$ to (\ref{dft}), we have%
\begin{equation*}
v\left( p\right) =av\left( p\right) +\left( c_{\kappa }(t)u+s_{\kappa
}(t)\nabla _{x}v\right) +s_{\kappa }(t)aH( p) \text{.}
\end{equation*}%
Now, the scalar product with $v( p)$ yields $a=1$, but the scalar product with $H( p)$
yields $a=0$ (indeed,  $s_{\kappa }(t)H(p) \neq 0$,
since $N$ is not totally geodesic and $s_{\kappa }$ does not vanish on $(0,T_{\kappa })$). This is a contradiction. Consequently, $%
(dF_+)_{(p,t)}$ is an isomorphism.

Finally, the smooth vector field that gives the foliation $\{\gamma
_{v(p)}(0,T_{\kappa })\,|\,p\in N\}$ of $U$ is given by $V=dF\circ
\left( 0,\tfrac{\partial }{\partial s}\right) \circ F^{-1}$.

\bigskip

``(b) $\Leftrightarrow $ (c)'': The equivalence of (b) and (c) follows from
the following linear algebra lemma, with $T=\left( \nabla v\right) _{p}$ and 
$W=v\left( p\right) ^{\perp }$. 
\end{proof}

\begin{lemma}
\label{lem:linalg}
Let $T:V\rightarrow V$ be a linear transformation whose image is contained
in the codimension one subspace $W$ \emph{(}in particular, 0 is an
eigenvalue of $T$\emph{)} and let $S=\left. T\right\vert _{W} \colon
W\rightarrow W$. Then $S$ has no real eigenvalues if and only if the only
real eigenvalue of $T$ is zero with algebraic multiplicity one.
\end{lemma}

\begin{proof}
The lemma is a consequence of the following two assertions: 
\begin{compactenum}[(a)]
\item A real number $\lambda \neq 0$ is an eigenvalue of $S$ if and only if $\lambda$ is an eigenvalue of $T$. 
\item The map $S$ has eigenvalue $0$ if and only if $T$ has eigenvalue $0$ with algebraic multiplicity greater than one.
\end{compactenum}
Both arguments are straightforward; we only write the details of (b).

Suppose that $0$ is an eigenvalue of $S$ with eigenvector $x \in W$. We may
assume that the image of $T$ is equal to $W$, since otherwise $\dim \operatorname{ker%
}(T) \geq 2$, and the proof is complete. Thus $x =Ty$ for some $y \in V$,
hence $y$ is a generalized eigenvector of $T$ (which must be linearly
independent from $x$) and so the eigenvalue $0$ of $T$ has algebraic
multiplicity at least $2$.

Conversely, suppose that $0$ is an eigenvalue of $T$ with algebraic
multiplicity at least $2$. Then the subspace $\left\{y \in V \mid
T^2y=0\right\}$ intersects $W$ nontrivially. If a nonzero vector $y \in W$
satisfies $T^2y = 0$, then either $y$ or $Ty$ is a nonzero vector in $\operatorname{%
ker}(T)$, and hence in $\operatorname{ker}(S)$.
\end{proof}

\section{Bifoliations and geodesic vector fields}

Here we prove Theorem \ref{thm:geodesic}, that geodesic vector fields are
bifoliating, but not all bifoliating vector fields are geodesic.

\begin{proof}[Proof of Theorem \protect\ref{thm:geodesic}]
Let $v$ be a smooth unit vector field on $N \subset M_\kappa$. To show part
(a), we assume that $v$ is geodesic and we verify that $v$ satisfies the
criterion of Theorem \ref{bifoliate}(b).

Suppose that $N$ has constant sectional curvature $k$, in particular $k>\kappa $. Let $p\in N$, let $x$ be a unit vector in $T_{p}N$ orthogonal to 
$v(p)$, and suppose that $\nabla _{x}v=\lambda x$, with $\lambda \in \mathbb{R}$.   As in the proof of (a) $\Rightarrow$ (b) of Theorem \ref{bifoliate},
but considering Jacobi fields on $N$ defined on $\left( 0,T_{k}\right)$
along geodesics in $N$ (instead of $M_{\kappa }$), we obtain that $\lambda
^{2}+k\leq 0$. Consequently, $\lambda ^{2}+\kappa \leq 0$ and so $v$ is
bifoliating by Theorem \ref{bifoliate}. This completes the proof of part (a).

We now prove part (b).  We begin by constructing a unit vector field on the unit sphere $S^{2k-1}$ which is
bifoliating in any ambient space but is not geodesic, obtained by perturbing the standard Hopf
fibration.

Let $J \colon \mathbb{R}^{2k} \to \mathbb{R}^{2k}$ denote the
standard almost complex structure, which we write explicitly as $%
J(p_1,\dots,p_{2k}) = (-p_2,p_1,\dots,-p_{2k},p_{2k-1})$.  We define $P \colon 
\mathbb{R}^{2k} \to \mathbb{R}^{2k}$ by $P(p_1,\dots,p_{2k}) =
(-p_4,p_3,-p_2,p_1,0,\dots,0)$ and the unit vector field 
\begin{equation*}
v_\varepsilon \colon S^{2k-1} \to S^{2k-1}, \ \ \ v_\varepsilon(p) = \frac{%
(J+\varepsilon P)(p)}{|(J+\varepsilon P)(p)|}. 
\end{equation*}
It is straightforward to check that for $0 \leq \varepsilon < 1$, $%
v_\varepsilon$ defines a smooth unit tangent vector field on $S^{2k-1}$.
Moreover, $v_0(p) = J(p)$ is tangent to the standard Hopf fibration on $%
S^{2k-1}$, and so by the computation in part (a), the restriction of $%
(\nabla v_0)_p$ to $v_0(p)^\perp$ has no real eigenvalues. (In fact, since
the restriction is equal to the restriction of the linear map $J$ itself, it
is easy to check that the eigenvalues are $\pm i$.) Now by continuity of the
roots of the characteristic polynomials and the compactness of $S^{2k-1}$,
the restriction of $(\nabla v_\varepsilon)_p$ to $v_\varepsilon(p)^\perp$
has no real eigenvalues for sufficiently small $\varepsilon > 0$, and
therefore such $v_\varepsilon$ are bifoliating by Theorem \ref{bifoliate}.

Observe that if a unit tangent vector field $v\colon S^{2k-1}\rightarrow
S^{2k-1}$ determines a great circle fibration, then $v(v(p))=-p$ for all $p$%
. We show next that $v_{\varepsilon }$ does not satisfy this condition at $%
p=e_{1}$ if $\varepsilon >0$. Since $v_{\varepsilon }$ is the normalization
of a linear map, it suffices to check that $(J+\varepsilon P)^{2}(e_{1})$
does not normalize to $-e_{1}$. This follows from the following computation: 
\begin{equation*}
(J+\varepsilon P)^{2}(e_{1})=(J+\varepsilon P)(e_{2}+\varepsilon
e_{4})=-(1+\varepsilon ^{2})e_{1}-2\varepsilon e_{3}.
\end{equation*}%
Hence $v_{\varepsilon }$ does not define a great circle fibration for $%
\varepsilon >0$.

Note that the proof can be repeated, with an appropriate scaling, to
construct an example on a sphere of any radius. Thus any sphere $N$ in any ambient space admits a bifoliating nongeodesic vector field.

It remains to consider the cases when $\kappa <0$ and $N$ is diffeomorphic
to $\mathbb{R}^{n}$, $n\geq 2$. If $N$ has constant negative sectional
curvature, by rescaling, we may suppose that $N$ is hyperbolic space with
curvature $-1$ (and so $\kappa <-1$). For $\varepsilon \in \mathbb{R}$ we
define on $N$ the unit vector field 
\begin{equation*}
w_{\varepsilon }\left( x_{0},\dots ,x_{n-1}\right) =\frac{x_{0}}{\sqrt{%
1+\varepsilon ^{2}\sin ^{2}x_{1}}}\left( 1,\varepsilon \sin x_{1},0,\dots
,0\right) \text{.}
\end{equation*}%
Then $w_{0}$ is geodesic (orthogonal to a foliation by parallel horospheres)
and by the proof of part (a), any real eigenvalue $\lambda $ of $\left(
\nabla w_{0}\right) _{p}$ satisfies $\lambda ^{2}\leq 1< - \kappa $ (actually, $%
\lambda ^{2}=1$). For $\varepsilon >0$ sufficiently small, the eigenvalue
condition is maintained but the vector field is no longer geodesic. For a
horosphere $N$, which is isometric to $\mathbb{R}^{n}$, a similar argument can be made for a perturbation
\[
u_\varepsilon(x_1,\dots,x_n) = \frac{1}{\sqrt{1+\varepsilon^2\sin^2(x_1)}}(1,\varepsilon \sin(x_1),0\dots,0),
\]
of the constant vector field $u_0$, since the eigenvalues of $(\nabla u_0)_p$ are zero for all $p$.
\end{proof}

\section{Contact forms and bifoliating vector fields on $S^3$}

Here we show that the $1$-form dual to a bifoliating vector field on a $3$-sphere $N
\subset S^4$ is a contact form. The proof is a simple computation.

\begin{proof}[Proof of Theorem \protect\ref{thm:contact}]
We study the contact condition for a unit $1$-form $\alpha $ on an oriented
Riemannian $3$-manifold $(N,g)$, with dual vector $v$. Consider $\xi %
\coloneqq \ker (\alpha )$ cooriented with $\alpha $, and consider $X,Y\in
\xi $. Then%
\begin{align*}
\alpha \wedge d\alpha (v,X,Y) & =d\alpha (X,Y) =g(\nabla _{X}v,Y)-g(X,\nabla
_{Y}v) \\
& =g(X,(\nabla v)^{T}Y)-g(X,(\nabla v)Y)  =g(X,((\nabla v)^{T}-\nabla v)Y).
\end{align*}

If $\alpha $ is not contact at some point $p$, then for all $x,y\in \xi _{p}$%
, $g_{p}(x,((\nabla v_{p})^{T}-\nabla v_{p})y)=0$. Thus $(\nabla
v_{p})^{T}-\nabla v_{p}$ restricts to the zero linear map on $\xi _{p}$, so $%
\nabla v_{p}=(\nabla v_{p})^{T}$ on $\xi _{p}$. Therefore $\left. \nabla
v_{p}\right\vert _{\xi _{p}}$ is symmetric and hence has real eigenvalues.
This implies by Theorem \ref{bifoliate}(b) that $v$ is not bifoliating.
\end{proof}

As mentioned following the statement of Theorem \ref{thm:contact}, the
same proof works in ambient space $\mathbb{R}^4$ with the additional
hypothesis that $(\nabla v)_p$ has rank $2$ for all $p$. The next example
shows that bifoliating vector fields on $S^3 \subset \mathbb{R}^4$ do not
necessarily satisfy this nondegeneracy condition; that is, there
exists a bifoliating vector field on a $3$-sphere $N \subset \mathbb{R}^4$ which does
not bifoliate the exterior of $N \subset S^4$ and whose dual is
not contact.  In particular, this vector field satisfies condition $(c)(i)$ of Theorem \ref{bifoliate} but not condition $(c)(ii)$.

\begin{proposition}
\label{prop:bizero} There exists a bifoliating unit vector field $v$ on $%
S^{3} \subset \mathbb{R}^4$ such that $\left( \nabla v\right) _{e^{it}}=0$
for all $t\in \mathbb{R}$.
\end{proposition}

\begin{proof}
We identify $S^3$ with unit quaternions and we represent the standard basis
elements as $\left\{ \mathbf{1}, \mathbf{i}, \mathbf{j}, \mathbf{k }\right\}$%
. The idea is to smoothly interpolate between a third-order bifoliating
vector field near $e^{\mathbf{i }t}$ and the standard Hopf vector field away
from $e^{\mathbf{i }t}$.

We define Fermi coordinates centered at the unit speed geodesic $\gamma(t)
= e^{\mathbf{i }t} = (\cos t, \sin t, 0, 0)$, on the whole $S^3$ except the
great circles $C_{\mathbf{1 }\mathbf{i}}$ and $C_{\mathbf{j }\mathbf{k}}$
determined by $\operatorname{span}\left\{ \mathbf{1}, \mathbf{i}\right\}$ and $\operatorname{%
span}\left\{ \mathbf{j}, \mathbf{k}\right\}$.

Let $E=\left\{ \left( t,r,s\right) \in \mathbb{R}^{3}\mid 0<r<\pi /2\right\}$
and define $\varphi \colon E \rightarrow S^{3}$ by 
\[
\varphi \left(
t,r,s\right) =\cos r~e^{\mathbf{i }t}+\sin r~e^{\mathbf{i }s} \mathbf{j}  = (\cos r \cos t, \cos r \sin t, \sin r \cos s, \sin r \sin s).
\]

Let $h:\left( -\frac{\pi }{2},\frac{\pi }{2}\right) \rightarrow \mathbb{R}$
be a $C^{\infty }$ odd strictly increasing function such that $h\left(
0\right) =h^{\prime }\left( 0\right) =0$ and $h\left( r\right) =\sin r$ for $%
1<r<\frac{\pi }{2}$, and let $g=\sqrt{1-h^{2}}$. Define the unit vector
field $v$ on $S^{3}$ by%
\begin{equation*}
v\left( \varphi \left( t,r,s\right) \right) = g(r) \mathbf{i }e^{\mathbf{i }%
t} + h(r) e^{\mathbf{i }s} \mathbf{k }= (-g(r)\sin t, g(r) \cos t, -h(r)
\sin s, h(r) \cos s). 
\end{equation*}
and $v(p) = \mathbf{i }p$ for $p \in C_{\mathbf{1 }\mathbf{i}} \cup C_{%
\mathbf{j }\mathbf{k}}$. Notice that $v$ coincides with the Hopf vector
field $q\mapsto \mathbf{i }q$ on an open neighborhood of $C_{\mathbf{j }%
\mathbf{k}}$, and so it is smooth there.  Smoothness at $C_{\one \ii}$ is not difficult to verify explicitly.  Moreover, the vector field is invariant
by rotations $R_{2\theta }(q) =e^{\mathbf{i }\theta }qe^{-\mathbf{i }\theta }
$ around $\gamma $ and transvections $T_{2\tau }(q) =e^{\mathbf{i}\tau }qe^{%
\mathbf{i}\tau }$ along $\gamma $. Therefore it suffices to check the
bifoliating property only when $t=s=0$. We compute the partial derivatives
of $\varphi$ at points $\rho \coloneqq (0,r,0)$: 
\begin{equation*}
\varphi _t(\rho) =\cos r~\mathbf{i}, \ \ \ \ \varphi _r(\rho) =-\sin r~%
\mathbf{1 }+\cos r~\mathbf{j}, \ \ \ \ \text{and} \ \ \ \ \varphi_s(\rho)
=\sin r~\mathbf{k};
\end{equation*}
these form a basis $\mathcal{B}$ of $T_{p}S^{3}$, where $p=\varphi (\rho)$.
We will compute $(\nabla v)_p$ in the basis $\mathcal{B}$.

Let $P:\mathbb{H}\rightarrow T_{p}S^{3}=p^{\perp }$ be the orthogonal
projection. A straightforward computation gives $P(\mathbf{1}) = \langle 
\mathbf{1}, \varphi_r(\rho) \rangle \varphi_r(\rho) = - \sin
r~\varphi_r(\rho)$ and similarly $P(\mathbf{j}) = \cos r~\varphi_r(\rho)$.

We compute 
\begin{eqnarray*}
\left( \nabla _{\varphi _{t}\left( \rho \right) }v\right) _{p} &=&\left. 
\tfrac{D}{dt}\right\vert _{0}v\left( \varphi \left( t,r,0\right) \right) \\
&=&P\left( \left. \tfrac{d}{dt}\right\vert _{0}v\left( \cos r~e^{\mathbf{i }%
t}+\sin r~\mathbf{j }\right) \right) \\
&=&P\left( \left. \tfrac{d}{dt}\right\vert _{0}g\left( r\right) \mathbf{i }%
e^{\mathbf{i }t}+h\left( r\right) \mathbf{k }\right) \\
&=&P\left( -g\left( r\right) \mathbf{1 }\right) \\
&=&g\left( r\right) \sin r~\varphi _{r}\left( \rho \right) \text{.}
\end{eqnarray*}%
We compute $\left( \nabla _{\varphi _{r}\left( \rho \right) }v\right) _{p}$
and $\left( \nabla _{\varphi _{s}\left( \rho \right) }v\right) _{p}$ in the
same way, obtaining%
\begin{equation*}
\left[ \left( \nabla v\right) _{p}\right] _{\mathcal{B}}=\left( 
\begin{array}{ccc}
0 & g^{\prime }\left( r\right) \sec r & 0 \\ 
g\left( r\right) \sin r & 0 & -h\left( r\right) \cos r \\ 
0 & h^{\prime }\left( r\right) \csc r & 0%
\end{array}%
\right) \text{,}
\end{equation*}%
for which the eigenvalues are $0$ and $\pm \sqrt{g\left( r\right) g^{\prime
}\left( r\right) \tan r-h\left( r\right) h^{\prime }\left( r\right) \cot r}$%
. Using the fact that $gg^{\prime }=-hh^{\prime }$, the radicand is equal to 
\begin{equation*}
-\frac{h(r)h^{\prime }(r)}{\sin r \cos r}, 
\end{equation*}
which is negative for $0 < r < \frac{\pi}{2}$.

On the other hand, a similar computation yields $(\nabla v)_{(0,0,0)}$ is
identically zero. Therefore $v$ is bifoliating and $(\nabla v)_{e^{\mathbf{i 
}t}} = 0$.
\end{proof}

\section{Outer billiards and preservation of volume}

Let $U$ be the exterior of a complete umbilic not totally geodesic
hypersurface $N$ in $M_\kappa$. Recall that if a smooth unit vector field $v$
bifoliates $U$, then a smooth invertible outer billiard map $B \colon U \to U$
is well-defined by (\ref{eqn:outerbilliarddef}).  Using the notation of the proof of Theorem \ref{bifoliate},  we write
$B=F_{+} \circ g \circ (F_{-})^{-1}$, where $g:N\times (-T_{\kappa },0)\rightarrow N\times (0,T_{\kappa })$ is the
smooth function given by $g(p,t)=(p,-t)$. 
The fact that $B$ is a diffeomorphism follows from the proof of Theorem \ref{bifoliate}.  To prove Theorem \ref{volume}, we compute the differential of $B$.

\begin{proof}[Proof of Theorem \protect\ref{volume}]
For $p\in N$, let $h\neq 0$ and let $\{u_{i}\,|\,i=1,\dots ,n\}$ be an
orthonormal basis of $T_{p}N$, such that $u_{1}=v\left( p\right) $ and $%
\mathcal{B}=\{u_{1},\dots ,u_{n},H(p)/h\}$ is a positively oriented
orthonormal basis of $T_{p}M_{\kappa }$. Now, for $s\in (-T_{\kappa
},T_{\kappa })$, let $\mathcal{B}_{s}$ be the basis of $T_{\gamma
_{v(p)}(s)}M_{\kappa }$ given by the parallel transport of $\mathcal{B}$
along $\gamma _{v(p)}$ between 0 and $s$. Since the parallel transport is an
isometry between the corresponding tangent spaces and preserves the
orientation, the bases $\mathcal{B}_{s}$ are positively oriented and
orthonormal as well.

Besides, we consider the basis of $T_{(p,s)}(N\times \mathbb{R})$ defined by 
\begin{equation*}
\mathcal{C}_{s}=\left\{ (u_{1},0),\dots ,(u_{n},0),\left( 0,\left. \tfrac{%
\partial }{\partial r}\right\vert _{s}\right) \right\} .
\end{equation*}%
Computing, we obtain 
\begin{equation*}
dF_{(p,s)}(u_{i},0)=J_{u_{i}}(s) \hspace{0.3cm}\text{and }\hspace{0.3cm}%
dF_{(p,s)}\left( 0,\left. \tfrac{\partial }{\partial r}\right\vert
_{s}\right) =\gamma _{v(p)}^{\prime }(s),
\end{equation*}%
where $J_{u_{i}}$ are as in (\ref{Jx}). Hence, the matrix of $dF_{(p,s)}$
with respect to the pairs of bases $\mathcal{C}_{s},\,\mathcal{B}_{s}$ is%
\begin{equation}
\lbrack dF_{(p,s)}]_{\mathcal{C}_{s},\mathcal{B}_{s}}=\left( 
\begin{array}{ccc}
1 & 0_{n-1}^{T} & 1 \\ 
s_{\kappa }(s)b & c_{\kappa }(s)\,I_{n-1}+s_{\kappa }(s)A & 0_{n-1} \\ 
s_{\kappa }(s)h & 0_{n-1}^{T} & 0%
\end{array}%
\right) \text{,}  \label{matrix dF}
\end{equation}%
where $A$ is the matrix of $(\nabla v)|_{v(p)^{\perp }}$ with respect to the
basis $\{u_{i}\,|\,i=2,\dots ,n\}$, $b$ is the column vector whose
entries are the coordinates of $\nabla _{v(p)}v$ with respect to the same basis, $0_m$ is the null column vector, and $I_m$ denotes the
identity $m \times m$ matrix.

Now, we fix $q\in U$ and let $p\in N$ and $t\in (0,T_{\kappa })$ such that $%
F_{-}(p,-t)=q$. We want to compute the determinant of $[dB_{q}]_{\mathcal{B}%
_{-t},\mathcal{B}_{t}}$. We observe that 
\begin{equation*}
dB_{q}=(dF_{+})_{(p,t)}\circ (dg)_{(p,-t)}\circ (dF_{-}^{-1})_{q}.
\end{equation*}

An easy computation shows that $\det [(dg)_{(p,-t)}]_{\mathcal{C}_{-t},%
\mathcal{C}_{t}}=-1$. Using (\ref{matrix dF}), for $0\neq s\in (-T_{\kappa
},T_{\kappa })$, we obtain 
\begin{eqnarray*}
\det \,[dF_{(p,s)}]_{\mathcal{C}_{s},\mathcal{B}_{s}} &=&s_{\kappa }(s)h\det
\,(c_{\kappa }(s)\,I_{n-1}+s_{\kappa }(s)A) \\
&=&h(s_{\kappa }(s))^{n}\det \,(\cot _{\kappa }(s)\,I_{n-1}+A).
\end{eqnarray*}%
Calling $P$ the characteristic polynomial of $(\nabla v)|_{v(p)^{\perp }}$,
since $\cot _{\kappa }(-s)=-\cot _{\kappa }(s)$, the above equality can be
written as 
\[
\det \,[dF_{(p,s)}]_{\mathcal{C}_{s},\mathcal{B}_{s}}=h(s_{\kappa
}(s))^{n}P(\cot _{\kappa }(-s)).
\]

Using that $(dF_{-}^{-1})_{q}=((dF_{-})_{(p,-t)})^{-1}$ and $s_{\kappa }$ is
an odd function, we obtain 
\[
\lbrack dB_{q}]_{\mathcal{B}_{-t},\mathcal{B}_{t}}=(-1)^{n+1}P(\cot
_{\kappa }(-t))/P(\cot _{\kappa }(t))\text{,}
\]
where the expression is well defined since $v$ is a bifoliating vector field
of $N$ (see part (c) of Theorem \ref{bifoliate}). Besides, the image of $%
(\nabla v)_{p}$ is contained in $v(p)^{\perp }$. Thus, the characteristic
polynomial $Q_{p}$ of $(\nabla v)_{p}$ satisfies $Q_{p}(s)=sP(s)$, for all $%
s $.  In consequence, since the image of $\cot _{\kappa }$ is open in the set of
real numbers, $B$ preserves the volume form if and only if $%
Q_{p}(s)=(-1)^{n}Q_{p}(-s)$ for all $s$.
\end{proof}

\begin{proof}[Proof of Corollary \protect\ref{cor:volume}]
For all parts of Corollary \ref{cor:volume}, it is useful to write 
\begin{align}  \label{eqn:qps}
Q_p(s) = s^n - \trace\left(\nabla v\right) _{p} s^{n-1} +
c_{n-2}(p)s^{n-2} + \cdots + c_2(p)s^2 + c_1(p)s,
\end{align}
where the term $c_0(p) =
\left( -1\right) ^{n}\det \left( \nabla v\right) _{p}$ vanishes since $v$ is unit.

To verify part (a), we apply Theorem \ref{volume} to equation (\ref{eqn:qps}%
). In particular, if $B$ preserves volume, the parity of each $Q_p(s)$
matches that of $n$. Thus the coefficient of $s^{n-1}$, namely $\trace\left(\nabla v\right) _{p}$, vanishes for all $p$, and so $\operatorname{div}(v)$ vanishes identically.

If additionally $n$ is even, then so is each $Q_p$, hence the linear
coefficient $c_1(p)$ is identically $0$.  Therefore $0$ is eigenvalue of $(\nabla v)_p$ with algebraic multiplicity $2$, so by item (b) in the proof of Lemma \ref{lem:linalg},  the restriction of $(\nabla v)_p$
to $v(p)^\perp$ has a zero eigenvalue. 

For part (b), let $n = 2$, so that $v(p)^\perp \subset T_{p}N$ has dimension
one for all $p \in N$.

``$(i) \Rightarrow (ii)$'': Let $u$ be a unit vector field on $N$ such that $%
u(p) \perp v(p) $ for all $p\in N$. Since $u$ is unit, $\langle \nabla
_{u}u,u\rangle =0$. Since $B$ is volume-preserving, $\nabla _{u}v=0$ by part
(a). Therefore we have 
\begin{equation*}
\left\langle \nabla _{u}u,u\right\rangle =0\text{\ \ \ \ and \ \ \ }%
\left\langle \nabla _{u}u,v\right\rangle =u\left\langle u,v\right\rangle
-\left\langle u,\nabla _{u}v\right\rangle =0\text{.}
\end{equation*}%
Since $\left\{ u,v\right\} $ is an orthonormal frame, $u$ is a geodesic
field and so $v$ is orthogonal to a line foliation of $N$.

``$(ii) \Rightarrow (i)$'': If $v$ is orthogonal to a geodesic foliation of $%
N$ given by a unit vector $u$, we have that%
\begin{equation*}
\left\langle \nabla _{u}v,u\right\rangle =u\left\langle u,v\right\rangle
-\left\langle \nabla _{u}u,v\right\rangle =0\text{.}
\end{equation*}%
Since $v$ is unit, $\left\langle \nabla _{w}v,v\right\rangle =0$ for all $w$%
. Then the matrix of $\nabla v$ with respect to the basis $\left\{
u,v\right\} $ is strictly upper triangular. Thus, $Q_p(s) =s^{2}$ for all $p$, and so 
$B$ preserves volume by part (a).

``$(i) \Leftrightarrow (iii)$'': Since $n=2$, equation (\ref{eqn:qps}) can
be written as 
\begin{equation*}
Q_p(s) = s^2 - \trace \left(\nabla v\right) _{p} s. 
\end{equation*}
The volume-preservation condition and the divergence-free condition both
correspond to the vanishing of the coefficient of $s$.

Now if $N \subset H^3$ is a horosphere, the claim is immediate since $N$,
with the intrinsic metric, is isometric to $\mathbb{R}^2$.


To verify item (c), observe that for $n=3$ equation (\ref{eqn:qps}) can be
written as 
\begin{equation*}
Q_p(s) = s^3 - \trace \left(\nabla v\right) _{p} s^2 + c_1 s,
\end{equation*}
and the volume-preservation condition and the divergence-free condition both
correspond to the vanishing of the coefficient of $s^2$. The final assertion
follows from the main result in \cite{GluckGu} (see also \cite%
{HarrisPaternain} and Proposition 1 in \cite{PeraltaSalas}), which states
that the only great circle fibrations of $S^3$ with volume-preserving flows are the
Hopf fibrations.
\end{proof}

\section{A Hopf-like bifoliating vector field on $\mathbb{R}^3 \subset H^4$}

Here we give an example of a unit vector field on a horosphere in the hyperbolic 
$4$-space bifoliating the exterior, and we find explicit periodic orbits, unbounded
orbits, and bounded nonperiodic orbits of the associated billiard map.

We consider the upper half-space model
 $H=\left\{ \left( p_{0},p_{1},\dots
,p_{n}\right) \in \mathbb{R}^{n+1}\mid p_{0}>0\right\} $ of the $\left(
n+1\right) $-dimensional hyperbolic space of constant curvature $-1$. For $0<h\leq 1$ define the
horospheres $H_h = \left\{ p \in H \mid p_0=h\right\}$. Let $v$ be a
bifoliating vector field on $N \coloneqq H_1$, so the associated billiard
map $B$ is well-defined on $U=\left\{p \in H \mid 0<p_0<1 \right\}$ and
preserves each horosphere $H_{h}$ for $0<h<1$. Keeping in mind that $H_h
\simeq \mathbb{R}^n$, we consider the following flow.

\begin{definition}
Let $v$ be a unit vector field on $\mathbb{R}^{n}$ and $\delta >0$. The 
\textbf{associated }$\left( v,\delta \right) $\textbf{-flow} is the discrete
flow on $\mathbb{R}^{n}$ generated by the map 
\begin{equation*}
f(v,\delta) \colon \mathbb{R}^n \to \mathbb{R}^n, \hspace{.25in}
f(v,\delta)(p) =p+\delta v(p). 
\end{equation*}
\end{definition}

Observe that for $\delta >0$ and $m\in \mathbb{N}$, both small, $%
f(v,\delta)^{m}(p)$ approximates 
the integral curve of $v$ which emanates from $p$.

It is convenient to write the billiard map $B$ in terms of $f$. Fix $0<h<1$
and let $r=\sqrt{1-h^{2}}$. We identify $H_{h}$ with $\mathbb{R}^{n}$ in the
obvious way.

\begin{proposition}
\label{Bh}The restriction $B_{h}$ of $B$ to $H_{h}$ equals 
\[
f(v\circ \left( \id-rv\right) ^{-1},2r).
\]
\end{proposition}

\begin{proof}
We first observe that $\left(\id - rv\right)^{-1}$ exists due to
Theorem \ref{bifoliate}; in particular, the Inverse Function Theorem applies
to $(\id - rv)$ because $-\frac{1}{r}$ is not an eigenvalue of $%
\nabla v$.

Let $q=\left( h,q^{\prime }\right) \in U$. By Theorem \ref{bifoliate}, there
exists $p=\left( 1,p^{\prime }\right) \in N$ such that $q=\gamma _{v\left(
p\right) }(s)$ for some $s<0$. Now, the image of $\gamma _{v\left( p\right) }
$ is the vertical semicircle centered at $\left( 0,p^{\prime }\right) $
tangent to $N$ at $p$ and containing $q$ (this was depicted in Figure \ref{fig:hyper} in Section \ref{sec:statement}).  Identifying $v(p)$ with $%
v(p^{\prime })$, we have $q^{\prime }=p^{\prime }-rv(p^{\prime })$. So $B(q)
=(h,p^{\prime }+rv( p^{\prime }))$ and hence, $B_{h}(q^{\prime })=p^{\prime
}+rv(p^{\prime })$, which equals 
\begin{equation*}
q^{\prime }+2rv(p^{\prime }) =f\left( v\circ \left( \id -r
v\right) ^{-1},2r\right)(q^{\prime }), 
\end{equation*}
since $\left( \text{id}-rv\right) ^{-1}( q^{\prime }) =p^{\prime }$.
\end{proof}

Now let $V$ be the Hopf vector field on $S^{3}$ given by $%
V\left(t,x,y,z\right) =\left( -x,t,-z,y\right)$, and consider the
stereographic projection 
\begin{equation*}
F\colon S^3 - \left\{(1,0,0,0)\right\} \to \mathbb{R}^3, \hspace{.25in}
F(t,x,y,z) = \frac{1}{1-t}(x,y,z). 
\end{equation*}
Let $v$ be the induced unit vector field on $\mathbb{R}^{3}$, which can be
written explicitly as 
\begin{equation*}
v(x,y,z) =\frac{1}{x^{2}+y^{2}+z^{2}+1}%
\left(x^{2}-y^{2}-z^{2}+1,2xy+2z,2xz-2y\right). 
\end{equation*}
It is invariant by rotations around the $x$-axis, that is, $v\circ R=dR\circ
v$ for any rotation $R$ fixing the $x$-axis. The circle $x=0$, $%
y^{2}+z^{2}=1 $ and the $x$-axis are images of integral curves of $v$.  See Figure \ref{fig:hopfstereo}.

\begin{figure}[ht!]
\centerline{
\includegraphics[width=4in]{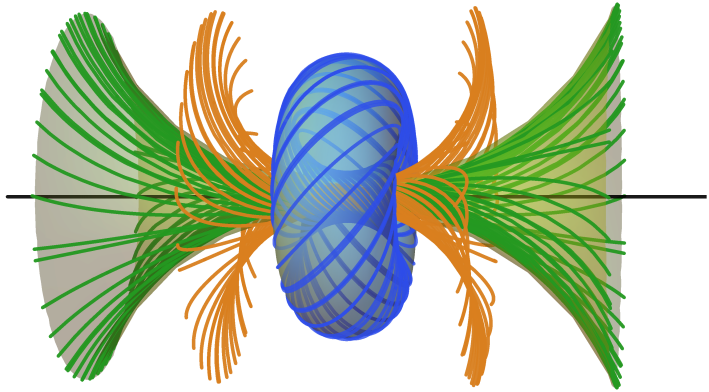}
}
\caption{Stereographic projection of the Hopf fibration; the straight line is the $x$-axis}
\label{fig:hopfstereo}
\end{figure}

We use the notation for the upper half space model established above; in
particular $H_h = \left\{p \in H \mid p_0 = h \right\}$, $N = H_1$, and $U =
\left\{p \in H \mid 0<p_0<1 \right\}$.

\begin{proposition}
\label{prop:hopfr3} The unit vector field $v$ defined above, considered on
the horosphere $N \simeq \mathbb{R}^{3} \subset H^{4}$, bifoliates $U$.
Moreover, for the associated billiard map $B \colon U \to U$ and fixed $h
\in (0,1)$, we have: 
\begin{compactenum}[(a)]
\item For $m \in \Z$, $B^m(h,x,0,0)=(h,x+2m\sqrt{1-h^{2}},0,0)$; in particular the orbit of $(h,x,0,0)$ is unbounded.
\item The map $B$ preserves the circle $C_h = \left\{(h,x,y,z) \mid y^2+z^2=2-h^2, x=0 \right\}$. 
\item The restriction of $B$ to $C_h$ is a rotation by angle $\theta_h=2\arctan\left(\sqrt{1-h^2}\right)$, and hence the orbits on $C_h$ are periodic if and only if $\theta_h$ is a rational multiple of $\pi$.
\end{compactenum}
\end{proposition}

\begin{proof}
To show that $v$ bifoliates $U$, it suffices by Theorem \ref{bifoliate} to
show that for all $p\in N$, the only real eigenvalue of $\left( \nabla
v\right) _{p}$ is $0$. Since $v$ is invariant by rotations about the $x$%
-axis, we only need to consider $p$ in the plane $z=0$. Now, the matrix of $%
\left( \nabla v\right) _{\left( x,y,0\right) }$ with respect to the
canonical basis is $\frac{2}{(1+x^2+y^2)^2}\,A$, where 
\begin{equation*}
A=\left( 
\begin{array}{ccc}
2xy^{2} & -2y(1+x^{2}) & 0 \\ 
y(1-x^{2}+y^{2}) & x(1+x^{2}-y^{2}) & 1+x^{2}+y^{2} \\ 
2xy & -(1+x^{2}-y^{2}) & x(1+x^{2}+y^{2})%
\end{array}%
\right) \text{.}
\end{equation*}%
Since $v$ is a unit vector field, $v$ is in the kernel of $A^{T}$, and so $v$
is an eigenvector of $A^{T}$ with associated eigenvalue $0$. The other two
eigenvalues of $A^{T}$ are $(1+x^{2}+y^{2})(x\pm i)$, with eigenvectors $%
(\mp iy,1\pm ix,x\mp i)$. Consequently, 0 is the only real eigenvalue of $%
\left( \nabla v\right) _{\left( x,y,0\right) }$ and so $v$ bifoliates $U$ by
Theorem \ref{bifoliate}.

Now using Proposition \ref{Bh} and the definition of $f$, we write 
\begin{equation*}
B_h(x,y,z) = f(v \circ (\text{id} - rv)^{-1},2r)(x,y,z) = (x,y,z) + 2r
v\left(\left(\text{id}-rv\right)^{-1}\left(x,y,z\right)\right), 
\end{equation*}
where $r = \sqrt{1-h^2}$.

To verify part (a), we use the fact that $(\id-rv)(x+r,0,0) =
(x,0,0)$, and we compute 
\begin{equation*}
B_h(x,0,0) = (x,0,0) + 2rv(x+r,0,0) = (x+2r,0,0) = (x+2\sqrt{1-h^2},0,0). 
\end{equation*}
The formula for $B^m_h(x,0,0)$ follows inductively.

To verify part (b), we use the fact that $(\id-rv)(0,1,0) =
(0,1,r)$, and we compute 
\begin{equation*}
B_h(0,1,r) = (0,1,r) + 2r v(0,1,0) = (0,1,-r). 
\end{equation*}
Therefore $B_h(0,1,r)$ has the same norm as $(0,1,r)$, and by the rotational
symmetry, this is true for all $(x,y,z) \in C_h$. Therefore $B_h$ is
invariant on $C_h$.

Part (c) follows from the observation that the angle $\theta_h$ subtending
the arc between $(0,1,r)$ and $(0,1,-r)$ is equal to $2\arctan(\sqrt{1-h^2})$%
. 
\end{proof}

\section{Further comments and questions}

We have classified bifoliating vector fields by an infinitesimal condition, and we have seen that prototypical examples of bifoliating vector fields are given by geodesic vector fields.  We conclude with several compelling questions regarding the topology and geometry of bifoliating vector fields and the dynamics of their associated outer billiards.

\begin{question}
\label{ques:defret}
Does the space of bifoliating vector fields on $N$ deformation retract to the space of geodesic vector fields on $N$?
\end{question}

In \cite{GW}, Gluck and Warner showed that the space of great circle fibrations of $S^3$ deformation retracts to its subspace of Hopf fibrations, so a positive answer to this question for $N = S^3 \subset S^4$ would provide a full topological classification of bifoliating vector fields.  Moreover, by Gray stability (see \cite{Geiges}, Theorem 2.2.2), it would give a tightness result for the contact structures induced by bifoliating vector fields.

On the other hand, the contact structure associated to a bifoliating vector field on $S^3$ naturally induces a symplectic structure $\omega$ on $S^3 \times (0,\infty) \simeq U$. 

\begin{question} Under what conditions on $v$ is the associated outer billiard map $B$ symplectic with respect to the induced symplectic structure $\omega$ on $U$?
\end{question}

Of course, the symplectic structure induced by the Hopf vector field is the standard one, and the corresponding outer billiard map is symplectic.  

We have seen that the volume-preservation of $B$ is related to the volume-preservation of $v$ itself, and we have seen that these conditions are equivalent for $n=2$ and $n=3$.

\begin{question} What is the relationship between volume-preservation of $B$ and the volume-preservation of $v$ in higher dimensions?
\end{question}

More specifically, in light of Theorem \ref{volume}, the volume-preservation condition of $B$, which is given by the vanishing of several coefficients of $Q_p(s)$, is a priori much stronger than the volume-preservation condition of $v$, which is given by the vanishing of a single coefficient of $Q_p(s)$.   However, we do not have an example of a divergence-free bifoliating vector field for which $B$ does not preserve volume.

On the other hand, we have seen from \cite{GluckGu} that a geodesic vector field on $S^3$ preserves volume if and only if it is Hopf. 

\begin{question} Is there a non-geodesic bifoliating vector field $v$ on $S^3$ such that $v$ (and hence $B$) preserves volume?
\end{question}

We are especially interested in understanding the dynamics of bifoliating outer billiard systems, for example in low dimensions.

\begin{question}  Does there exist a bifoliating vector field on a horosphere in $H^3$ such that the associated outer billiard map has a periodic orbit?
\end{question}

By Proposition \ref{Bh}, each orbit of $B_{h}$ describes a polygonal line in $H_{r} \simeq \R^2$, with edges of length $2r$. As $h$ tends to $1$, these lines approach in a certain sense the integral curves of $v$, which cannot be closed. We believe that the condition on the eigenvalues of $\nabla v$ prevents these polygonal lines from being closed.

\vspace{0.5cm}

\noindent Yamile Godoy\newline
\noindent Conicet\,-\,Universidad Nacional de C\'ordoba\newline
\noindent CIEM\thinspace -\thinspace FaMAF\newline
\noindent Ciudad Universitaria, 5000 C\'{o}rdoba, Argentina\newline
\noindent yamile.godoy@unc.edu.ar

\vspace{0.5cm}

\noindent Michael Harrison\newline
\noindent Institute for Advanced Study\newline
\noindent 1 Einstein Drive\newline
\noindent Princeton, NJ 08540, US\newline
\noindent mah5044@gmail.com

\vspace{0.5cm}

\noindent Marcos Salvai\newline
\noindent Conicet\,-\,Universidad Nacional de C\'ordoba\newline
\noindent CIEM\thinspace -\thinspace FaMAF\newline
\noindent Ciudad Universitaria, 5000 C\'{o}rdoba, Argentina\newline
\noindent salvai@famaf.unc.edu.ar

\end{document}